\documentclass[11pt,reqno]{amsart}

\usepackage{amssymb}
\usepackage{latexsym}
\usepackage{amsbsy}
\usepackage{amsfonts}
\usepackage{appendix}
\usepackage{pstricks}
\usepackage{pstcol,pst-fill,pst-grad}
\usepackage{enumitem}% Pour interrompre et reprendre des énumérations

\newtheorem{theorem}[equation]{Theorem}

\newtheorem{proposition}[equation]{Proposition}
\newtheorem{lemma}[equation]{Lemma}
\newtheorem{corollary}[equation]{Corollary}
\newtheorem{definition}[equation]{Definition}

\theoremstyle{definition}
\newtheorem{auxremark}[equation]{Remark}
\newenvironment{remark}{%
  \begin{auxremark}%
  }{%
   \hfill$\diamondsuit$%
    \end{auxremark}
  }

\newtheorem{example}[equation]{Example}

\numberwithin{equation}{section}

\newdimen\AAdi%
\newbox\AAbo%
\def\AAk#1#2{\setbox\AAbo=\hbox{#2}\AAdi=\wd\AAbo\kern#1\AAdi{}}%

\makeatletter
\def\eqlabel#1{\def\@currentlabel{#1}}

\def\formula#1{\def\@tempa{#1}\let\@tempb\theequation\def\theequation{%
\hbox{#1}}\def\@currentlabel{(\theequation)}$$}
\def\endformula{\leqno\hbox{(\@tempa)}$$\@ignoretrue\let\theequation\@tempb}

\def\given{\hskip5\p@\relax\vrule\@width.4\p@\hskip5\p@\relax}

\newcommand{\open}[1]{%
\par\normalfont\topsep6\p@\@plus6\p@\trivlist\item[\hskip\labelsep\itshape#1%
\@addpunct{.}]\ignorespaces}

\DeclareRobustCommand{\close}[1]{%
  \ifmmode % if math mode, assume display: omit penalty etc.
  \else \leavevmode\unskip\penalty9999 \hbox{}\nobreak\hfill
  \fi
  \quad\hbox{$#1$}}
\makeatother

\newlength{\toskip}

\def\<{\langle}
\def\>{\rangle}

\def \R {{\mathbb R}}

\def \E {{\mathbb E}}

\def \L {{\mathbb L}}

\def \Var {\textrm{Var}}
\def \Ent {\textrm{Ent}}

\newcommand*{\dmu}{\,d\mu}% pour ne pas retaper l'espace.

\begin{document}
%%%%%%%%%%%%%%%%%%%%%%%%%%%%%%%%%%%%%%%%%%%%%%%%%%%%%%%%%%%%%%%%%%%%%%%
%%%%%%%%%%%%%%%%%%%%%%%%%%%%%%%%%%%%%%%%%%%%%%%%%%%%%%%%%%%%%%%%%%%%%%%
\date{\today}

\title[Lyapunov and Hitting.]{Hitting times, functional inequalities, Lyapunov conditions and uniform ergodicity.}

 \author[P. Cattiaux]{\textbf{\quad {Patrick} Cattiaux $^{\spadesuit}$ \, \, }}
\address{{\bf {Patrick} CATTIAUX},\\ Institut de Math\'ematiques de Toulouse. CNRS UMR 5219. \\
Universit\'e Paul Sabatier,
\\ 118 route
de Narbonne, F-31062 Toulouse cedex 09.} \email{cattiaux@math.univ-toulouse.fr}

\author[A. Guillin]{\textbf{\quad {Arnaud} Guillin $^{\diamondsuit, \clubsuit}$}}
\address{{\bf {Arnaud} GUILLIN},\\ Laboratoire de Math\'ematiques, CNRS UMR 6620, Universit\'e Blaise Pascal,
avenue des Landais, F-63177 Aubi\`ere.} \email{guillin@math.univ-bpclermont.fr}

\maketitle

 \begin{center}

 \textsc{$^{\spadesuit}$  Universit\'e de Toulouse}
\smallskip

\textsc{$^{\diamondsuit}$ Universit\'e Blaise Pascal}
\smallskip

\textsc{$^{\clubsuit}$ Institut Universitaire de France}
\smallskip

\end{center}

\begin{abstract}
The use of Lyapunov conditions for proving functional inequalities was initiated in \cite{BCG}. It was shown in \cite{BBCG,CGZ} that there is an equivalence between a Poincar\'e inequality, the existence of some Lyapunov function and the exponential integrability of hitting times. In the present paper, we close the scheme of the interplay between Lyapunov conditions and functional inequalities by
\begin{itemize}
\item showing that strong functional inequalities are equivalent to Lyapunov type conditions;
\item showing that these Lyapunov conditions are characterized by the finiteness of generalized exponential moments of hitting times.
\end{itemize}
We also give some complement concerning the link between Lyapunov conditions and integrability property of the invariant probability measure and as such transportation inequalities, and we show that some ``unbounded Lyapunov conditions'' can lead to uniform ergodicity, and coming down from infinity property.
\end{abstract}
\bigskip

\bigskip

\bigskip

\section{Introduction}

Let $D$ be some smooth open domain in $\mathbb R^d$. In this paper, we will mainly consider the differential operator defined for smooth functions $f \in C^\infty(D)$ by $$Lf=\frac 12\sum_{ij}(\sigma \, \sigma^*)_{ij}(x) \, \frac{\partial^2_{ij}f}{\partial x_{ij}}+\sum_i b_i(x)\frac{\partial_i f}{\partial x_i} \, ,$$ where 
$\sigma$ is an $\R^{d\times m}$ smooth and bounded (for simplicity $C_b^\infty(\bar D)$) matrix field and $b$ a $C^\infty(\bar D)$ vector field. \\ We may see $L$ as the infinitesimal generator of a diffusion process associated to the stochastic differential equation (SDE) 
$$dX_t=\sigma(X_t)dB_t+b(X_t)dt \quad , \quad X_0=x \, ,$$
where $B_t$ is an usual $\R^m$-Brownian motions when $D=\mathbb R^d$, or to the reflected SDE $$dX_t=\sigma(X_t)dB_t+b(X_t)dt + dN_t , \quad  \int_0^t \, \mathbf 1_{\partial D}(X_s) dN_s = N_t \quad X_0=x \, ,$$ if $D$ is some smooth subdomain. \\ The domain $\mathcal D(L)$ of $L$ (viewed as a generator) is thus some extension of the set of smooth and compactly supported functions $C_c^\infty(\bar D)$ such that the normal derivative $\frac{\partial f}{\partial n}$ vanishes  on $\partial D$ (if $\partial D$ is non void). This corresponds to normal reflection or to a Neumann condition on the boundary. We also define  $P_t$ the associated semi-group $$P_tf(x)=\mathbb E_x(f(X_t))$$ which is defined for bounded functions $f$. \\ In order to use classical results in PDE theory we will also assume that $L$ is uniformly elliptic, i.e. $$\sigma \, \sigma^* \, \geq \, a \, Id$$ in the sense of quadratic forms for some $a>0$, or more generally that $L$ is uniformly strongly hypo-elliptic in the sense of Bony (see \cite{bony}) and that the boundary is non characteristic. For details we refer to \cite{catbord}. 
\medskip

\noindent We will also assume (though it should be a consequence of some of our assumptions) that there exists a probability measure $$\mu(dx) = e^{-V(x)} \, dx$$ which is an invariant measure for the process (or the semi-group) i.e. for all bounded and smooth function $f \in \mathcal D(L)$, $$\int \, Lf \, d\mu =0 \quad \textrm{ or equivalently for all $t$,} \quad \E_\mu(f(X_t)) = \int f \,  d\mu \, .$$ $P_t$ then extends to a contraction semi-group on $\mathbb L^p(\mu)$ for $1\leq p \leq +\infty$. We shall say that $\mu$ is symmetric, or that $P_t$ is $\mu$ symmetric if for smooth $f$ and $g$ in the domain of $L$, $$\int f \, Lg \, d\mu = \int \, Lf \, g \, d\mu \, .$$ The standard example of $\mu$-symmetric semi-group is obtained for $\sigma= \sqrt 2 \, Id$ and $b= - \nabla V$ (provided $V$ is smooth enough). In all cases our ellipticity assumptions imply that this measure is unique and ergodic.
\medskip

Among the most fascinating recent developments at the border of analysis and probability theory, a lot of work has been devoted to the study of the relationship between 
\begin{itemize} \item geometric properties of the measure $\mu$, for instance concentration properties, \item functional inequalities (the study of weighted Sobolev or Orlicz-Sobolev spaces associated to $L$ and $\mu$) like the Poincar\'e (Wirtinger) inequality or the Gross logarithmic Sobolev inequality, \item  transportation inequalities like the $T_2$ Talagrand's inequality, \item the rate of convergence to equilibrium for the semi-group $P_t$ in various functional spaces, $\mathbb L^2(\mu)$, Orlicz spaces related to $\mu$, \item the rate of convergence of its dual $P_t^*$ (i.e. the distribution of the process at time $t$) in total variation  or in Wasserstein distance, \item the existence of Lyapunov functions, \item and finally some properties of the stochastic process $X_.$ in particular existence of general moments for hitting times of some subsets (for instance the control of how the process comes down from infinity in the ultracontractive situation). \end{itemize}

We refer to the monographs of Davies \cite{Dav}, Ledoux \cite{ledoux01}, Wang \cite{Wbook} and Bakry-Gentil-Ledoux \cite{bakgl}, the surveys by Gross \cite{grossbook} and Ledoux \cite{ledoux}, the collective book \cite{logsob} and the papers \cite{aida,bakry-emery,BCR1,BCR2,BCR3,BK08,BL97,bobkov-gotze,bgl,cat4,CatGui1,CGG,LO00,OV} among many others, for the first four items.\\ For the last three items we refer to the monographs of Hasminskii \cite{Has} and Meyn-Tweedie \cite{MT} and the papers \cite{CKexit,DFG,DMT,GM,MT2,MT3,V97} among many others.
\medskip

The link between both approaches was done in \cite{BCG} for the first time, up to our knowledge. It was extended in \cite{BBCG,CatGui3,CGWW,CGW,CGW2}. One can see the (now outdated) survey \cite{CatGui4}. 
\medskip

\noindent To be a little bit more precise, let us recall the following result from \cite{CGZ} Theorem 2.3  (also see \cite{kulik} for a similar statement)
\begin{theorem}\label{thmpoinc}
Consider the following properties
\begin{enumerate}
\item[(HP1)] There exists a Lyapunov function $W$, i.e. a smooth function $W:D \to \R$, s.t. $W\geq w >0$, and there exist a constant $\lambda > 0$ and an open connected bounded subset $U$ such that $\frac{\partial W}{\partial n}= 0$ on $\partial D$ and 
$$ LW \, \leq \, - \,
\lambda \, W \, +  \mathbf 1_{\bar{U}} \, .$$
\item[(HP2)] There exist an open
connected bounded subset $U$ and a constant $\theta>0$ such that for all $x$, $$
\E_x\left(e^{\theta \, T_U}\right) < + \infty \, ,$$ where $T_U$ denotes the hitting time of $U$.
\item[(HP3)] The process is geometrically ergodic, i.e. there exist constants $\beta >0$ and $C>0$ and a
function $W\geq 1$ belonging to $\L^1(\mu)$ such that for all $x$
$$
 \Vert P_t(x,.) - \mu \Vert_{TV} \, \leq \, C \, W(x) \, e^{- \beta \, t} \, ,$$ where $P_t(x,.)$ denotes the distribution of $X_t$ (when $X_0=x$) and $\Vert .\Vert_{TV}$ denotes the total variation distance. 
\item[(HP4)] $\mu$ satisfies a Poincar\'e inequality, i.e. there exists a constant $C_P(\mu)$ such that
for all smooth $f \in \mathcal D(L)$,
$$\Var_\mu(f) \, \leq \, C_P(\mu) \, \mathcal E(f) \, ,$$ where $$\mathcal E(f)\,  =  \, \int \,  -Lf \, f \, d\mu \, = \, \frac 12 \, \int \, |\sigma \, \nabla f|^2 \dmu \, .$$
\item[(HP5)]
There exists a constant $\lambda_P(\mu)$ such that for all $f\in \L^2(\mu)$,$$
\Var_\mu(P_tf) \, \leq \, e^{-
\, \lambda_P(\mu) \, t} \, \Var_\mu(f) \, .$$
\end{enumerate}
Then $(HP5) \Leftrightarrow (HP4)$, $(HP4) \Rightarrow (HP2)$, $(HP2) \Leftrightarrow (HP1)$ and $(HP1) \Rightarrow (HP3)$. Actually $(HP4)$ implies $(HP2)$ for all (non-empty) open connected and bounded subset $U$.\\ If in addition $\mu$ is \underline{symmetric}, then $$(HP1) \Leftrightarrow (HP2) \Leftrightarrow (HP3) \Leftrightarrow (HP4) \Leftrightarrow (HP5) \, .$$
\end{theorem}
When $\mu$ is not symmetric, examples are known (kinetic diffusions) where $(HP1)$ (hence $(HP3)$) is satisfied but $(HP4)$ is not.  In the case of kinetic diffusions, it is evident that $(HP4)$ cannot hold as the Dirichlet form is degenerate. It is however also possible to build "monster" diffusion where the invariant probability measure has some polynomial tail but the diffusion (with identity diffusion matrix) may converge exponentially fast and thus a Lyapunov condition holds. \\

 The first equivalence is well known (see \cite{logsob}), the second one is a simple application of Ito's calculus and PDE results (see \cite{CGZ} and \cite{catbord}), the final implication is a consequence of the Meyn-Tweedie theory. The implication $(HP4) \Rightarrow (HP2)$ is shown in \cite{CGZ} by using the deviation results for the occupation measure obtained in \cite{CatGui2} (using a beautiful deviation result obtained in \cite{wu1}). We shall see in the next section another much more direct approach. Finally, in the symmetric case the converse implications are obtained by using the method in \cite{BBCG}.
\medskip

Some extensions of this theorem to polynomial ergodicity are discussed in \cite{CGZ} in connexion with weak Poincar\'e inequalities. A deeper study of this situation is done in \cite{ll,lll1,lll2}.
\bigskip

The questions we shall address in the present paper are not concerned with weakening but with reinforcing of the assumptions, that is, does it exist similar results as in Theorem \ref{thmpoinc} if we replace the Poincar\'e inequality by stronger inequalities, for instance $F$-Sobolev inequalities in the spirit of \cite{Aid98,BCR1,BCR3} ?
\medskip

Partial answers are known: an $F$-Sobolev inequality is equivalent to exponential stabilization in some Orlicz space (see \cite{RZ}) and in the symmetric situation, reinforced Lyapunov conditions imply super-Poincar\'e inequalities or $F$-Sobolev inequalities (see \cite{CGWW,CGW,CGW2}). Recently, Liu (\cite{liu1,liu2}) proposed some new ideas in order to directly link Lyapunov conditions on one hand, concentration properties or functional inequalities on the other hand. Though some aspects of his proofs are a little bit obscure for us, we shall follow his main idea in order to deduce a Lyapunov condition from a functional inequality, and then get equivalent results in terms of hitting times. This yields the following result written here for the logarithmic Sobolev inequality which is the best known $F$-Sobolev
\begin{theorem}\label{thmls}
Assume that $D$ is not bounded. Consider the following properties
\begin{enumerate}
\item[(HLS1)] There exists a Lyapunov function $W$, i.e. there exists a smooth function $W:D \to \R$, with  $W\geq w >0$, and there exist constants $\lambda > 0$ and $b>0$ such that $\frac{\partial W}{\partial n}= 0$ on $\partial D$ and 
$$ LW(x) \, \leq \, - \,
\lambda \, |x|^2 \, W(x) \, +  b \, .$$
\item[(HLS1')] There exists a Lyapunov function $W$, i.e. there exists a smooth function $W:D \to \R$, with $W\geq w >0$, and there exist constants $\lambda > 0$ and $b>0$ such that $\frac{\partial W}{\partial n}= 0$ on $\partial D$ and 
$$ LW(x) \, \leq \, - \,
\lambda \, V(x) \, W(x) \, +  b \, .$$
\item[(HLS2)] There exist an open
connected bounded subset $U$ and a constant $\theta>0$ such that for all $x$, $$
\E_x\left(\exp\left(\int_0^{T_U} \, \theta \, |X_s|^2 \, ds\right)\right) < + \infty \, ,$$ where $T_U$ denotes the hitting time of $U$.
\item[(HLS2')] There exist an open
connected bounded subset $U$ and a constant $\theta>0$ such that for all $x$, $$
\E_x\left(\exp\left(\int_0^{T_U} \, \theta \, V(X_s) \, ds\right)\right) < + \infty \, ,$$ where $T_U$ denotes the hitting time of $U$.
\item[(HLS4)] $\mu$ satisfies a logarithmic-Sobolev inequality, i.e. there exists a constant $C_{LS}(\mu)$ such that for all smooth $f \in \mathcal D(L)$,
$$Ent_\mu(f^2) \, := \, \int \, f^2 \, \ln\left(\frac{f^2}{\int \, f^2 \, d\mu}\right) \, d\mu \, \leq \, C_{LS}(\mu) \, \mathcal E(f) \, ,$$ where $$\mathcal E(f)\,  =  \, \int \,  -Lf \, f \, d\mu \, = \, \frac 12 \, \int \, |\sigma \, \nabla f|^2 \dmu \, .$$
\item[(HLS5)]
There exists a constant $C_E(\mu)$ such that for all $f^2\in \L\ln \L(\mu)$ s.t. $\int f^2 d\mu=1$,$$
Ent_\mu(P_t(f^2)) \, \leq \, e^{-
\, C_E(\mu) \, t} \, \Ent_\mu(f^2) \, .$$
\end{enumerate}
Then 
\begin{itemize}
\item[1)] \, $(HLS5) \Leftrightarrow (HLS4)$, $(HLS4) \Rightarrow (HLS1)$ and $(HLS2) \Leftrightarrow (HLS1)$. \\
\item[1')] \, Assume that $V$ goes to infinity at infinity and that there exists some $a>0$ such that $\mu(e^{aV})<+\infty.$\\ Then $(HLS5) \Leftrightarrow (HLS4)$, $(HLS4) \Rightarrow (HLS1')$ and $(HLS2') \Leftrightarrow (HLS1')$.
\end{itemize}
Actually $(HLS4)$ implies in both cases $(HLS2)$ or $(HLS2')$ for all open nice subset $U$.
\bigskip

Assume in addition that $\mu$ is \underline{symmetric} and that $\sigma.\sigma^*$ is uniformly elliptic.
\begin{itemize}
\item[2')] \, Assume that $V$ goes to infinity at infinity, that $|\nabla V(x)|\geq v >0$ for $|x|$ large enough and that there exists some $a>0$ such that $\mu(e^{aV})<+\infty.$ \\ Then $(HLS1') \Leftrightarrow (HLS2') \Leftrightarrow (HLS4) \Leftrightarrow (HLS5)$.\\
\item [2)] \, Assume the curvature condition $Ric + Hess V \geq - C > - \infty$ where Ricci and Hess are related to the riemanian metric defined by $\sigma$. \\ Then $(HLS1) \Leftrightarrow (HLS2) \Leftrightarrow (HLS4) \Leftrightarrow (HLS5)$.
\end{itemize}
\end{theorem}  
Except  $(HLS5) \Leftrightarrow (HLS4)$ which is well known, we will prove this Theorem (and actually more general results) in section \ref{secls}. Part of this result can be extended to general $F$-Sobolev inequalities, this is done in section \ref{secFs}.
\medskip

The next two sections are devoted to the rate of convergence of $P_t(\nu,.)$, the distribution at time $t$ of the process with initial distribution $\nu$, to the invariant distribution $\mu$ for the total variation distance. In section \ref{seclp} Theorem \ref{thmconvvartot} we get the following: under some natural assumptions on $\mu$, almost any $F$-Sobolev inequality combined with the Poincar\'e inequality provides an exponential convergence $$\Vert P_t(\nu,.) - \mu \Vert_{TV} \leq \, C(\nu) \, e^{- \, \beta \, t} \, ,$$ for all $\nu$ absolutely continuous w.r.t. $\mu$ such that $\frac{d\nu}{d\mu}$ belongs to $\mathbb L^p(\mu)$ for some $p>1$. The remarkable fact here is that $\beta$ does not depend on the integrability property (the $p$) of $d\nu/d\mu$. If this and more general results were proved in \cite{CatGui3}, the proof given here is particularly simple and understandable. \\ In the next section \ref{secomdown} we study $\mathbb L^\infty$ properties of the Lyapunov functions, in relation with the property of ``coming down from infinity'' for the process. In particular we show that if the ultra boundedness property of the semi-group implies the ``coming down from infinity'' property for the process, the converse is not true. All these notions are particularly relevant in the study of quasi-stationnary distributions.
\medskip

Finally in the last section we directly rely general Lyapunov condition to the existence of some exponential moments for the measure $\mu$, extending the results in \cite{liu1}.
\bigskip

\section{Back to the Poincar\'e inequality.}\label{secpoinc}

As we said in the introduction, we shall give here a new direct proof of $(HP4) \Rightarrow (HP1)$ in Theorem \ref{thmpoinc}.
\medskip

\begin{theorem}\label{thmpoinlyap}
Assume that $\mu$ satisfies a Poincar\'e inequality with constant $C_P(\mu)$. Then for all open subset $A$, there exists a smooth Lyapunov function $W \in \mathcal D(L)$ i.e. a smooth function satisfying $W\geq w>0$ on $A^c$ and $LW \leq -c \, W$ on $A^c$ with $$c = \mu(A) \min\left(\frac{1}{4 C_P(\mu)} \, , \, \frac{1}{8}\right) \, .$$ 
\end{theorem}
It is easily seen that if $A$ is smooth an bounded, we can modify $W$ to get $(HP1)$ (with a $\lambda$ smaller than the $c$ in the theorem).
\begin{proof}
Let start with a simple lemma. 
\begin{lemma}\label{lemequiva}
Assume that $\mu$ satisfies a Poincar\'e inequality with constant $C_P(\mu)$. Then for all subset $A$ such that $\mu(A)>0$ and all $f$ in $H^1(\mu)=\mathcal D(\mathcal E)$ it holds $$ \int \, f^2 \, d\mu \leq   \, \frac{2 C_P(\mu)}{\mu(A)} \, \mathcal E(f) \, + \, \frac{4}{\mu(A)} \, \int_A \, f^2 \, d\mu \, .$$
\end{lemma}
\begin{proof}
Using Poincar\'e inequality and the elementary $$(a+b)^2 \leq (1+\lambda) a^2 + \left(1 +\frac 1 \lambda\right) \, b^2 \quad \textrm{ for all $\lambda >0$} $$ we can write
\begin{eqnarray*}
\int \, f^2 \, d\mu &\leq& C_P(\mu) \, \mathcal E(f) \, + \, \mu^2(f) \\ &\leq& C_P(\mu) \, \mathcal E(f)\, + \mu^2(f \mathbf 1_A + f \mathbf 1_{A^c})\\ &\leq& C_P(\mu) \, \mathcal E(f)\, + (1+\lambda) \, \mu^2(f \mathbf 1_A) + \frac{1+\lambda}{\lambda} \,  \mu^2(f \mathbf 1_{A^c}) \\ &\leq& C_P(\mu) \, \mathcal E(f)\, + (1+\lambda) \, \mu(A) \,  \int_A \, f^2 \, d\mu + \frac{1+\lambda}{\lambda} \,  \mu(f^2) \mu(A^c)
\end{eqnarray*}
where we have used Cauchy-Schwartz in the final inequality. Hence provided
$$\mu(A)(1+\lambda) - 1 \, > \, 0 \, ,$$ we have obtained $$\int \, f^2 \, d\mu \leq \frac{\lambda \, C_P(\mu)}{\mu(A)(1+\lambda)-1} \, \mathcal E(f)\, + \frac{\lambda (1+\lambda) \, \mu(A)}{\mu(A)(1+\lambda)-1} \,  \int_A \, f^2 \, d\mu \, .$$ The result follows by choosing $\lambda = 2 (1-\mu(A))/\mu(A)$.
\end{proof}
Now define 
\begin{equation}\label{eqphipoinc}
\phi(x)= - c + \, \mathbf 1_A(x) \, ,
\end{equation}
and introduce for all smooth $u \in \mathcal D(L)$, $$Hu = -Lu+\phi \, u \, .$$ On one hand, it holds $$\mu(u.Hu) \leq \mathcal E (u) + \, \int_A \, u^2 \, d\mu \, .$$ On the other hand, applying the previous lemma we have
\begin{eqnarray*}
\mu(u.Hu) &=&  \mathcal E (u) + \, \int_A \, u^2 \, d\mu - c \mu(u^2)\\ &\geq&  \mathcal E (u) + \, \int_A \, u^2 \, d\mu - c \, \left(\frac{2 C_P(\mu)}{\mu(A)} \, \mathcal E(u) \, + \, \frac{4}{\mu(A)} \, \int_A \, u^2 \, d\mu\right)\\ &\geq& \frac 12 \left(\mathcal E (u) + \, \int_A \, u^2 \, d\mu\right)
\end{eqnarray*}
if we choose $$c = \mu(A) \min\left(\frac{1}{4 C_P(\mu)} \, , \, \frac{1}{8}\right) \, .$$ 
\medskip

Now we will linearize $\mu(u.Hu)$. If $v \in \mathcal D(L)$ and $u \in H^1(\mu)$, $\mathcal H(u,v)= \mu(u. Hv)$ is well defined, and using an integration by parts (or the Green-Rieman formula since the normal derivative of $v$ at the boundary vanishes) can be written as a (non necessarily symmetric) bilinear form on $H^1(\mu)$. It is easily seen that this bilinear form $\mathcal H$ is continuous on $H^1(\mu)$ equipped with the (usual) hilbertian norm $\left(\mathcal E(u) + \int \, u^2 \, d\mu\right)^{\frac 12}$, hence equipped with the hilbertian norm $\left(\mathcal E(u) + \int_A \, u^2 \, d\mu\right)^{\frac 12}$ which is equivalent according to Lemma \ref{lemequiva}. But according to what precedes, $\mathcal H$ is also coercive for this norm. \\
 Hence, we may apply the Lax-Milgram theorem which tells us that for any $g \in H^1(\mu)$ there exists some $v \in H^1(\mu)$ such that for all $u$, $\mathcal H(u,v)=\langle u,g \rangle$. We will use this result with $g \equiv 1$.\\ First of all, the previous relation with $u$ compactly supported in $D$ shows that $Hv=g$ in $\mathcal D'(D)$, so that thanks to ellipticity (or hypo-ellipticity), $v \in C^\infty(D)$ and satisfies $Hv=g$ in the usual sense in $D$. When $D=\mathbb R^d$ this is enough. Otherwise, since the boundary is non characteristic, $v$ admits sectional traces on $\partial D$ of any order (see \cite{catbord} Theorem 4.6) and using the results in \cite{catbord} section 4, one can see that $v \in C^\infty(\bar D)$ and satisfies $\partial_n v=0$ on $\partial D$. Since $Hv \in \mathbb L^2(\mu)$, it follows that $v \in \mathcal D(H)$. \\ As a routine, defining $v_-=\min(v,0)$, one can check using integration by parts or the Green-Rieman formula, that $\mu(v_-.Lv)=\mathcal E(v_-)$ so that, using the previous lower bound we obtain,
\begin{eqnarray*}
\mu(v_-.g)&=& \mathcal H(v_-,v)= \mathcal \mu(v_-.Hv) \, = \, \mathcal \mu(v_-.Lv) \, + \, \mu(\phi. v . v_-) = - \, \mathcal E(v_-) \, - \, \mu(\phi.  v^2_-) \\ &\leq& \, - \, \frac 12 \, \left(\mathcal E (v_-) + \, \int_A \, v_-^2 \, d\mu\right)
\end{eqnarray*}
and $v_-=0$ $\mu$ since $g>0$. So $v\geq 0$ almost surely. \\ One should now use the maximum principle but we prefer use Ito's formula. Assume that $A$ is open and bounded. Since $\phi= -c$ on $A^c$, we get for any $x \in A^c$, 
\begin{eqnarray*}
v(x)&=&\E_x(v(X_{t\wedge T_A}) + \E_x \left(\int_0^t \, \mathbf 1_{s\leq T_A} \, (Hv-\phi \, v)(X_s) \, ds \right)\\ &\geq&  \, \E_x(t\wedge T_A) > \, 0 \, 
\end{eqnarray*}
since $T_A>0$ if $x \in A^c$. Replacing $A$ by $A_\varepsilon= \{y, d(y,A)<\varepsilon\}$ we may finally let $t$ go to infinity, use the fact that $v$ is bounded from below by a positive constant $v(A,\varepsilon)$ on $\partial A_\varepsilon$ which is compact and obtain that $v(x)\geq v(A,\varepsilon)$ for all $x\in A_\varepsilon^c$ using the previous inequality (which actually furnishes exactly the minimum principle). For a general open set $A$ just take the intersection with a large ball to get a bounded subset.
\end{proof}
\begin{remark}\label{remdomL}
As soon as we know the existence of a Lyapunov function satisfying $(HP1)$ it immediately follows using Ito's formula with the function $(t,y) \mapsto e^{\lambda t}  \, W(y)$ that for all $x$, $$W_U(x)=\E_x(e^{\lambda \, T_U}) < +\infty \, ,$$ and conversely, if the exponential moment is finite, $W_U$ is a Lyapunov function (see \cite{CGZ}). \\ Also notice that the proof of Theorem \ref{thmpoinlyap} furnishes a Lyapunov function in $H^1(\mu)$, hence in $\mathcal D(L)$ since $Lv=(\mathbf 1_A-c)v -1$ implies $Lv \in \mathbb L^2(\mu)$. Hence the previous $W_U$ belongs to $\mathcal D(L)$. Conversely if $W_U$ is finite, $LW_U=-\lambda W_U$ in $U^c$. Replacing $\lambda$ by $\lambda/2$ if necessary, we may assume that $W_U$ is in $\mathbb L^2(\mu)$ so that $LW_U$ is square integrable too (at least in $U^c$). If $U$ is relatively compact it is easy to see that one can modify $U$ and $W_U$ to get a smooth function everywhere that belongs to $\mathcal D(L)$. 
\end{remark}
%Assume that $L+L^*$ satisfies our assumptions, in particular this is true in the uniformly elliptic situation (of course $L^*$ denotes the adjoint of $L$ in $\mathbb L^2(\mu)$). Then, we may build some diffusion process with generator $(L+L^*)/2$ and as soon as the hitting time of some $U$ as before has a finite exponential moment, we know that $\mu$ satisfies a Poincar\'e inequality with $$\mathcal E(f) = \int \, - \, \frac{(L+L^*)f}{2} \, f \, d\mu \, = \, \int \, - \, Lf  \, f \, d\mu \, .$$ It follows that for the original diffusion process associated with $L$, the hitting time of $U$ also has a finite exponential moment. Hence, as for the $\L^2$ contraction of the semi-group, the ``worse'' case is the symmetric one, i.e. once exponential integrability is true in the symmetric case it extends to any stationary situation.\medskip

\section{The logarithmic Sobolev inequality.}\label{secls}

We start with an analogue of Theorem \ref{thmpoinlyap}.
\begin{proposition}\label{proplslyap}
Assume that $\mu$ satisfies the logarithmic Sobolev inequality $(HLS4)$. Let $h$ be a non-negative continuous function such that $b=2 \mu(e^h) <+\infty$. For $\varepsilon >0$, define $U_\varepsilon(h)=\{(1-\varepsilon)h >b\}$. \\ Then there exists a Lyapunov function $W \in \mathcal D(L)$ such that $W(x) \geq w_\varepsilon >0$ on $U_\varepsilon(h)$ and $$(HLh) \quad LW \leq - \frac{\varepsilon}{2 \, C_{LS}} \, h \, W \quad \textrm{ on $U_\varepsilon(h)$.} $$
\end{proposition}
\begin{proof}
We follow and modify the proof in \cite{liu2}. Assume that $h$ is a non-negative function such that $\mu(e^h)<+\infty$. We follow the proof of Theorem \ref{thmpoinlyap} and define, for $2 \, \rho \leq \frac{1}{C_{LS}}$ so that $2\rho \, Ent_\mu(u^2) \leq \mathcal E(u)$,
\begin{equation}\label{eqphils}
\phi(x)= \rho \, (-  \, h(x)+ \, b) \, ,
\end{equation}
with $b=2 \mu(e^h)$ and introduce for all smooth $u \in \mathcal D(L)$, $$Hu = -Lu+\phi \, u \, .$$ On one hand, it holds $$\mu(u.Hu) \leq \mathcal E (u) + \, \rho \, b \, \mu(u^2) \, .$$ On the other hand, applying this time Young's inequality and LSI, we get for a smooth $u \in \mathcal D(L)$,
\begin{eqnarray*}
\mu(u.Hu) &=&  \mathcal E (u) + \, \rho \, b \, \mu(u^2) \, - \, \rho \, \mu(h u^2) \\ &\geq&  \mathcal E (u) + \, \rho \, b \, \mu(u^2) -  \, \rho \, \mu(u^2) \, \mu\left(e^h - \frac{u^2}{\mu(u^2)} + \frac{u^2}{\mu(u^2)} \, \ln \left(\frac{u^2}{\mu(u^2)}\right)\right)\\ &\geq& \mathcal E (u) + \, \rho \, b \, \mu(u^2) -  \, \rho \, \frac b2 \, \mu(u^2) \, + \, \rho \, \mu(u^2) \, - \, \rho \, Ent(u^2) \\ &\geq&  \frac 12 \, \left(\mathcal E (u) + \, \rho \, b \, \mu(u^2)\right) \, .
\end{eqnarray*}
We can then follow the proof of Theorem \ref{thmpoinlyap}, and thus apply again the Lax-Milgram theorem to  get the existence of a non-negative smooth function $v \in H^1(\mu)$ satisfying $$Lv = \phi \, v \, - \, 1 \, =  \, - \, 1 - \, \rho \, (b-h) \, v \, .$$ If in addition $$h(x) > b \quad , \forall x \in U^c$$ we obtain that $v$ is bounded from below by a positive constant in $U_\varepsilon^c$. The proof is completed.
\end{proof}
\noindent This result is in particular interesting when $D$ is not bounded and $h$ goes to infinity at infinity. Two cases are mainly relevant, due to the converse statements we will prove below
\begin{corollary}\label{corlslyap}
Assume that $\mu$ satisfies the logarithmic Sobolev inequality $(HLS4)$ and that $D$ is not bounded. Then
\begin{itemize}
\item[1)] \, for all $x_0 \in D$, there exists a Lyapunov function $W \in \mathcal D(L)$ with $W(x)\geq w>0$ for all $x\in D$ satisfying $$LW(x) \leq - \lambda \, d^2(x,x_0) \, W(x) \, + \, b \, ,$$ for some $\lambda$ and $b$ strictly positive;
\item[2)] \, if in addition $V$ goes to infinity at infinity and $e^{aV} \in \mathbb L^1(\mu)$ for some $a>0$, there exists a Lyapunov function $W \in \mathcal D(L)$ with $W(x)\geq w>0$ for all $x\in D$ satisfying $$LW(x) \leq - \lambda \, V(x) \, W(x) \, + \, b \, ,$$ for some $\lambda$ and $b$ strictly positive.
\end{itemize}
We may replace $b$ by $b \mathbf 1_A$ for some well chosen bounded subset $A$ of $D$.
\end{corollary} 

\begin{proof}
In case 2), just apply the previous proposition with $h=a|V|$ and modify $W$ in the corresponding level set $A$ of $V$. For case 1), recall that the logarithmic-Sobolev inequality implies that there exists some $c>0$ such that $\mu(e^{cd^2(.,x_0)})<+\infty$ and conclude as before with $h=cd^2(.,x_0)$.
\end{proof}
\noindent Introducing the process $$dH_t = h(X_t) \, dt$$ and applying Ito's formula to $(H,x) \mapsto e^{\frac{\varepsilon}{2 \, C_{LS}} \, H}$ it is easy to show that $(HLh)$ implies for all $x$,
\begin{equation}\label{eqexpgen}
W_{h,\varepsilon}(x) \, = \, \mathbb E_x \left(\exp\left(\int_0^{T_{U_{\varepsilon(h)}}} \, \frac{\varepsilon}{2 \, C_{LS}} \, h(X_s) \, ds\right)\right) < +\infty \, .
\end{equation}
Conversely, if $W_{h,\varepsilon}(x)$ is finite for all $x$, using the arguments in \cite{catbord} one can prove that it satisfies $(HLh)$ with an equality instead of an inequality. Notice that once again we may apply the arguments in Remark \ref{remdomL}.
\medskip

To complete the proof of Theorem \ref{thmls}, it remains to look at the converse statements in the symmetric situation.\\ The first case is the case $h(x)=c d^2(x,x_0)$. Statement 2) in Theorem \ref{thmls} under the curvature assumption is proved in \cite{CGW} using transportation inequalities. An alternative method of proof was recently proposed in \cite{liu2} (with some points to be corrected). For the second case $h=aV$ we will use the results in \cite{CGWW} based on super-Poincar\'e inequalities.
\begin{proposition}\label{propconvV}
Assume that $\mu$ is symmetric and that $\sigma . \sigma^*$ is uniformly elliptic. Assume in addition that $V$ goes to infinity at infinity, that $|\nabla V(x)|\geq v >0$ for $|x|$ large enough and that $e^{aV} \in \mathbb L^1(\mu)$ for some $a>0$.\\ If there exists a Lyapunov function $W$ with $W(x)\geq w>0$ for all $x\in D$, $\frac{\partial W}{\partial n}= 0$ on $\partial D$ and 
satisfying $$LW(x) \leq - \lambda \, V(x) \, W(x) \, + \, b \, ,$$ for some $\lambda$ and $b$ strictly positive, then $\mu$ satisfies a logarithmic-Sobolev inequality.
\end{proposition}
\begin{proof}
We follow the method in \cite{CGWW} Theorem 2.1 (itself inspired by \cite{BBCG}). Let $A_r=\{V \leq r\}$. For $r_0$ large enough and some $\lambda'<\lambda$ we have $$LW(x) \leq - \lambda' \, V(x) \, W(x) \, + \, b \, \mathbf 1_{A_{r_0}} \, ,$$ so that we may assume that $$\frac{LW}{W}(x) \leq - \, \lambda \, V(x) \quad \textrm{ for $x \in A_r^c$ and all $r$ large enough.}$$ 
Denote by $M = \sup (-V)$. We have for $s \leq s_0$ and $r>r_0$,
\begin{eqnarray*}
\int \, f^2 \, d\mu &=& \int_{A_r} f^2 \, d\mu \, + \, \int_{A^c_r} f^2 \, d\mu \\
&\leq&e^M\left(1+\frac{b}{\lambda r_0}\right) \, \int_{A_r} f^2 \, dx \, +\,\frac{1}{\lambda r}\int \lambda V(x) \,f^2d\mu,\\
 &\leq& e^M\left(1+\frac{b}{\lambda r_0}\right) \, \int_{A_r} f^2 \, dx \, + \, \frac{1}{\lambda \, r} \, \int f^2 \, \left(\frac{-LW}{W}\right) \, d\mu \\ &\leq& e^M\left(1+\frac{b}{\lambda r_0}\right) \, \left(s \, \int_{A_r} |\nabla f|^2 \, dx \, + \, \frac{C}{s^{d/2}} \, \left(\int_{A_r} |f| \, dx\right)^2\right) \, + \, \frac{1}{\lambda \, r} \, \int |\sigma. \nabla f|^2 \, \, d\mu \, .
\end{eqnarray*}
The first part of the last bound is obtained by using (3.1.4) in \cite{CGWW} (it is here that we are using the assumption on $|\nabla V|$), while the second bound is obtained using integration by parts or the Green-Rieman formula (see \cite{CGWW} (2.2)). Using uniform ellipticity we thus obtain, denoting $c=e^M\left(1+\frac{b}{\lambda r_0}\right)$
\begin{equation}\label{eqsuperP}
\mu(f^2) \, \leq \, \left(\frac{s \, c}{a} \, + \, \frac{1}{\lambda \, r}\right) \, \int |\sigma. \nabla f|^2 \, \, d\mu \, + \, C \, s^{-d/2} \, c \, e^{2r} \,  \left(\int \,  |f| \, d\mu\right)^2 \, .
\end{equation}
Hence choosing $r=c'/s$ we get the following super-Poincar\'e inequality for small $s$, $$\mu(f^2) \, \leq \, s \, \int |\sigma. \nabla f|^2 \, \, d\mu \, + \, C' \,  e^{\tilde c/s} \,  \left(\int \,  |f| \, d\mu\right)^2 \, ,$$ which is known to be equivalent to a defective logarithmic Sobolev inequality (see the introduction of \cite{CGWW}). But the Lyapunov condition being stronger than $(HP1)$, we know that $\mu$ satisfies a Poincar\'e inequality, hence using Rothaus lemma, that it satisfies a (tight) log-Sobolev inequality.

\end{proof}
\begin{remark}\label{remjerison}
With some additional effort one should replace (3.1.4) in \cite{CGWW} directly by a similar statement with $\sigma.\nabla f$ instead of $\nabla f$, even in the strongly hypo-elliptic case, replacing the arguments in \cite{CGWW} by the Jerison and Sanchez-Calle estimates for such operators, up to a modification of the power $s^{-d/2}$ replaced by $s^{-m}$ where $m$ depends on the dimension of the graded Lie algebra. We do not want to go further into details here, that is why we choosed to only consider the uniformly elliptic situation.
\end{remark}
\bigskip

\section{F-Sobolev inequalities.}\label{secFs}

We will extend the results of the previous section to general $F$-Sobolev inequalities introduced by Aida (\cite{Aid98}) and studied in \cite{Wbook,BCR1,BCR3}. Actually we will not be as complete as for the log-Sobolev inequality, because for general functions $F$ instead of the logarithm, results are much more intricate. In particular the reader will find in \cite{RZ} convergence results in Orlicz spaces (replacing $(HLS5)$ we shall not give here.

We are mainly interested here with the following version of (defective) $F$-Sobolev inequalities for a nice $F$ defined on $\mathbb R^+$: $$(HFS4defect) \quad \int \, f^2 \, F\left(\frac{f^2}{\int f^2 \, d\mu}\right) \, \leq \, C_F(\mu) \, \mathcal E(f) \, + \, D_F \, \mu(f^2) \, .$$
When $D_F=0$ one say that the inequality is tight and simply denote it by $(HFS4)$. The relationship between $F$-Sobolev and super-Poincar\'e inequalities is due to Wang (\cite{Wbook} Theorem 3.3.1 and Theorem 3.3.3). Recall some basic facts on these inequalities
\begin{proposition}\label{propFsob}
We have:
\begin{itemize}
\item \, (see \cite{Wbook}.) \, A super-Poincar\'e inequality $$\mu(f^2) \leq s \, \mathcal E(f) + \beta(s) \, (\mu(|f|))^2$$ implies $(HFS4defect)$ for $$F(u)=\frac 1u \, \int_0^u \, \xi(t/2) \, dt \quad , \quad \xi(t) = \sup_{a>0} \left(\frac 1a \, - \, \frac{\beta(a)}{ta}\right) \, .$$
\item \, (see \cite{Wbook}.) \, If $(HFS4defect)$ holds true for $F$ such that $\int_.^{+\infty} \, \frac{1}{uF(u)} \, du < +\infty$, then the semi-group $P_t$ is ultra-bounded, i.e. for all $t>0$ there exists $C(t)$ such that, $$\Vert P_tf \Vert_\infty \, \leq C(t) \, \Vert f\Vert_{\mathbb L^1(\mu)}$$ so that if in addition a Poincar\'e inequality holds, $$\Vert P_tf -\mu(f)\Vert_\infty \, \leq M \, e^{-C t}\, \Vert f-\mu(f)\Vert_{\mathbb L^2(\mu)}$$ for some $C>0$ (one can replace $\mathbb L^2$ by any $\mathbb L^p$ for $p>1$ just changing $C$ using interpolation results, see e.g. \cite{CatGui3,CGR}).
\item \, (see \cite{BCR1} lemma 8.) \, If $F(1)=0$, $F$ is $C^2$ in a neighborhood of $1$ and $2F'(1)+F''(1)=c>0$, $(HFS4)$ implies the Poincar\'e inequality $(HP4)$ with $C_P(\mu)=1/(2c)$. 
\item \, (see \cite{BCR1} Remark 22.) \, If $F\geq 0$ and $F(u)\geq c >0$ for $u\geq 2$, then $(HFS4)$ implies the Poincar\'e inequality $(HP4)$.
\item \, (see \cite{BCR1} lemma 9.) \, If $F$ is concave, non-decreasing, growths to infinity and satisfies $F(1)=0$ and $uF'(u)\leq M$, then $(HFSdefect)$ and the Poincar\'e inequality imply $(HFS4)$.
 \end{itemize}
\end{proposition}
We thus have two situations: either $F \leq \log$ (interpolating between Poincar\'e and log-Sobolev), in which case (with additional structural conditions on $F$) $F$-Sobolev inequalities are equivalent to an exponential convergence in some Orlicz space (\cite{RZ}), we still have some Rothaus(-Orlicz) lemma allowing us to tight a defective $F$-Sobolev inequality and a lot of additional properties connected with Orlicz hyperboundedness, concentration and isoperimetry (see \cite{BCR1,BCR3}), or $\int _.^{+\infty} \, \frac{1}{uF(u)} \, du < +\infty$ in which case exponential convergence holds in $\mathbb L^\infty$, with a very small gap between both classes of $F$.
\medskip

We will now prove the analogue of Proposition \ref{proplslyap}. To this end we need to introduce some convexity notions. 
\begin{definition}\label{defconvex}
Assume that $u \mapsto uF(u)=G(u)$ is convex. We define $G^*$ as the Fenchel-Legendre dual function of $G$ i.e. $G^*(u) = \sup_{t>0} \, (ut - G(t))$.\\ For instance if $F(u)$ behaves like $\ln_+^\beta(u)$ at infinity for some $\beta >0$, then $G^*(t)$ behaves like $\beta \, t^{(\beta-1)/\beta} \, e^{t^{\frac 1\beta}}$ at infinity (see \cite{BCR1} subsection 7.1 to see the correct $F$ to be chosen).
\end{definition}
\begin{proposition}\label{propFSlyap}
Assume that $\mu$ satisfies the F-Sobolev inequality $(HFS4defect)$ and that $G(s)=sF(s)$ is convex. Let $h$ be a non-negative continuous function such that $b=2 (D_F+\mu(G^*(h))) <+\infty$. For $\varepsilon >0$, define $U_\varepsilon(h)=\{(1-\varepsilon)h >b\}$. \\ Then there exists a Lyapunov function $W \in \mathcal D(L)$ such that $W(x) \geq w_\varepsilon >0$ on $U_\varepsilon(h)$ and $$LW \leq - \frac{\varepsilon}{2 \, C_{F}} \, h \, W \quad \textrm{ on $U_\varepsilon(h)$.} $$
\end{proposition}
\begin{proof}
The proof mimic the one of proposition \ref{proplslyap} with the following modifications: take $\phi=\rho(-h+b)$ with $b=2 (D_F+\mu(G^*(h)))$ and $\rho \, C_F=\frac 12$, use Young's inequality $st \leq G(s)+G^*(t)$.
\end{proof}
We thus obtain, for a particular choice of $F$:
\begin{theorem}\label{thmFS}
Assume that $D$ is not bounded, that $V$ goes to infinity at infinity and that $e^{aV} \in \mathbb L^1(\mu)$ for some $a>0$. Consider the following properties
\begin{itemize}
\item[(HFS1)]  There exists a Lyapunov function $W$, i.e. there exists a smooth function $W: D \to \mathbb R$ with $W\geq w >0$, and there exist constants $\lambda >0$ and $b>0$ such that $\frac{\partial W}{\partial n}= 0$ on $\partial D$ and 
$$LW(x) \leq - \, \lambda \, |V|^\beta(x) \, W(x) \, + \, b \, .$$
\item[(HFS2)] There exist an open
connected bounded subset $U$ and a constant $\theta>0$ such that for all $x$, $$
\E_x\left(\exp\left(\int_0^{T_U} \, \theta \, |V|^\beta(X_s) \, ds\right)\right) < + \infty \, ,$$ where $T_U$ denotes the hitting time of $U$. 
\item [$(HFS\beta)$] $\mu$ satisfies $(HFS4defect)$ with $F(s)= \ln_+^\beta(s)$.
\end{itemize}
Then $(HFS\beta) \Rightarrow (HFS1)$ and $(HFS1) \Leftrightarrow (HFS2)$.
\medskip

If in addition $\mu$ is symmetric, $\sigma.\sigma^*$ is uniformly elliptic and $|\nabla V(x)|\geq v >0$ for $|x|$ large enough, then $$(HFS\beta) \Leftrightarrow (HFS1)\Leftrightarrow (HFS2) \, .$$ For $\beta \leq 1$ we may replace $(HFS\beta)$ by its tight version.
\end{theorem}
\begin{proof}
For the first part we use the previous proposition with $h=a|V|^\beta$ for some $a$ small enough. In the symmetric situation, we mimic the proof of proposition \ref{propconvV} yielding a super-Poincar\'e inequality with $\beta(s)=e^{c/s^{\frac 1\beta}}$ hence the corresponding defective $F$-Sobolev inequality using Proposition \ref{propFsob}. But $(HFS1)$ implies $(HP1)$ hence a Poincar\'e inequality and we can use the final statement of Proposition \ref{propFsob} to get a tight version when $\beta \leq 1$.
\end{proof}
We know in particular (see \cite{BCR1}) that for $V(x)=|x|^\alpha$, $\alpha\geq 1$, $\mu$ satisfies $(HFS\beta)$ with $$\beta=2\left(1-\frac 1\alpha\right) \, .$$ The main reason for writing Theorem \ref{thmFS} only in the case $F \sim \ln_+^\beta$ is the converse part $(HFS1) \Rightarrow (HFS\beta)$ for which the argument is easy since we have an  explicit expression of $G^*$. Of course Proposition \ref{propFSlyap} contains  much more general situations. Here is one which will be useful in the sequel
\begin{theorem}\label{thmgoodinteg}
Assume that $D$ is not bounded, that $\mu$ satisfies both a Poincar\'e inequality and the F-Sobolev inequality $(HFS4defect)$, that $G(s)=sF(s)$ is convex, non decreasing at infinity and that $(G^*)^{-1}$ (the inverse function of $G^*$) growths to infinity at infinity. \\ Then for all $x_0 \in D$ there exist $a$ and $\theta$ two positive constants, such that, defining $$h(x)=(G^*)^{-1}(e^{ad(x,x_0)}) \, ,$$ we have for all $x$ and all non empty, open and bounded subset $U$, $$W_{\theta,U,h}(x)=
\E_x\left(\exp\left(\int_0^{T_U} \, \theta \, h(X_s) \, ds\right)\right) < + \infty \, ,$$ where $T_U$ denotes the hitting time of $U$. Actually $W_{\theta,U,h} \in \mathbb L^1(\mu)$.
\end{theorem}
\begin{proof}
Since $\mu$ satisfies a Poincar\'e inequality, it is known that there exists $a>0$ such that $$\mu(G^*(h))=\mu(e^{ad(x,x_0)})<+\infty \, .$$ In addition $h$ goes to infinity at infinity so that its level sets are compact. It remains to apply Proposition \ref{propFSlyap} to get the Lyapunov function $W_{\theta,U,h}$ and consequently the result. As usual since $LW_{\theta,U,h} + \theta h W_{\theta,U,h} \leq 0$ outside of a compact set, we get that $h \, W_{\theta,U,h}$ is integrable and since $h$ goes to infinity that $W_{\theta,U,h}$ is integrable too.
\end{proof}
\bigskip

\section{$\mathbb L^p$ geometric ergodicity and functional inequalities.}\label{seclp}

Come back to the geometric ergodicity property $(HP3)$. If we replace the initial distribution $\delta_x$ by some initial probability distribution $\nu$, we have $$\Vert P_t(\nu,.) - \mu \Vert_{TV} \leq \, C \, e^{- \, \beta \, t} \, ,$$ provided
\begin{equation}\label{eqW}
(LW\nu) \qquad \qquad W \in \mathbb L^1(\nu) \, .
\end{equation}
If $\nu$ is absolutely continuous w.r.t. $\mu$ and $\frac{d\nu}{d\mu} \in \mathbb L^p(\mu)$ a sufficient condition is thus $$(LWq) \qquad \qquad W \in \mathbb L^q(\mu)$$ for $\frac 1p + \frac 1q =1$. It is thus interesting to study the property $(LWq)$. 

\noindent As shown in \cite{CGZ}, once $(HP1)$ is satisfied, $(LW1)$ is satisfied too. It follows that for some $\theta >0$, $W_\theta(x)= \mathbb E_x(e^{\theta T_U})$ is finite, hence satisfies $LW_\theta = - \theta W_\theta$ on $U^c$, so that enlarging a little bit $U$ (say $U_\varepsilon$) we can modify the previous $W_\theta$ in $U_\varepsilon$ in order to get a new Lyapunov function, still denoted $W_\theta$ for simplicity satisfying $$LW_\theta \leq \, - \theta W_\theta + b \mathbf 1_{\bar U} \, .$$ Hence $W_\theta \in \mathbb L^1(\mu)$. It follows that for every $p \geq 1$, defining $W_{\theta,p}(x)= \mathbb E_x(e^{\frac {\theta}{p} T_U})$, we have first that $W_{\theta,p} \in \mathbb L^p(\mu)$, second that (after similar modifications) $W_{\theta,p}$ is a Lyapunov function with $\theta$ replaced by $\theta/p$. Since $\bar U$ is compact, these modifications do not modify the integrability properties of $W_{\theta,p}$. Hence we have obtained
\begin{proposition}\label{propqconv}
If $(HP1)$ is satisfied, for all $p>1$, one can find another Lyapunov function (associated to a different $\lambda$ and a different $U$) $W_p \in \mathbb L^p(\mu)$. Hence there exists some $\beta_p$ such that as soon as $\frac{d\nu}{d\mu} \in \mathbb L^q(\mu)$ with $\frac 1p + \frac 1q =1$,$$\Vert P_t(\nu,.) - \mu \Vert_{TV} \leq \, C(\nu) \, e^{- \, \beta_p \, t} \, .$$  
\end{proposition}

\noindent In the symmetric case the situation is better understood. Indeed, $(HP1)$ implies the geometric ergodicity in $\mathbb L^2(\mu)$ $(HP5)$, so that using Riesz-Thorin interpolation theorem in an appropriate way (see \cite{CGR}) we have that provided $\frac{d\nu}{d\mu} \in \mathbb L^q(\mu)$ for some $1<q\leq 2$, $$\left \Vert P_t\left(\frac{d\nu}{d\mu}\right) -1 \right \Vert_{\mathbb L^q(\mu)} \leq K_q \, e^{- \, \frac{(q-1)}{q} \, \lambda_P(\mu) \, t} \, \left\Vert \frac{d\nu}{d\mu} - 1\right\Vert_{\mathbb L^q(\mu)} \, .$$
\noindent In this situation we thus have geometric convergence for a stronger topology.
\medskip

But the discussion preceding Proposition \ref{propqconv} furnishes a stronger result. Indeed $(HP1)$ yields $(HP3)$ 
\begin{equation}\label{eqH}
H(x) = \Vert P_t(x,.) - \mu \Vert_{TV} \leq \, C \, W(x) \, e^{- \, \beta \, t} \, ,
\end{equation}
 so that $H(x)$ converges to $0$ for all $x$ at a geometric rate. But we may replace $W$ by $W_p$ and $\beta$ by $\beta_p$ and get that actually $H$ converges to $0$ in all $\mathbb L^p(\mu)$ for $1\leq p < +\infty$, with a geometric rate depending on $p$. 
\medskip

Assume from now on that $\mu$ satisfies some $F$-Sobolev inequality, for some smooth $F$. Does it improve the previous results ? In this situation, the rate of convergence to equilibrium in total variation distance was studied in \cite{CatGui3}. The results we proved in the previous sections allow us to give a new and substantially simpler proof of some results contained in \cite{CatGui3}. Here is a result in this direction:
\begin{theorem}\label{thmconvvartot}
Under the assumptions of Theorem \ref{thmgoodinteg}, there exists $\lambda >0$ such that for all non-empty, open and bounded set $U$, $$W(.) = \mathbb E_.\left( e^{\lambda \, T_U}\right) \, \in \, \bigcap_{1\leq q<+\infty} \, \mathbb L^q(\mu) \, .$$ It thus follows that there exists some $\beta>0$ such that for all $\nu$ absolutely continuous w.r.t. $\mu$ such that $\frac{d\nu}{d\mu}$ belongs to $\mathbb L^p(\mu)$ for some $p>1$, $$\Vert P_t(\nu,.) - \mu \Vert_{TV} \leq \, C(\nu) \, e^{- \, \beta \, t} \, .$$ Actually $H$ defined in \eqref{eqH} converges to $0$ in all the $\mathbb L^p$'s with a rate $e^{-\beta \, t}$.
\end{theorem}
\begin{proof}
Let $h$ as in Theorem \ref{thmgoodinteg}. The level sets $H_R=\{h \leq R\}$ of $h$ are smooth (since $F$ is smooth) compact sets. Denote by $T_R$ the hitting time of $H_R$. For $R$ large enough, $H_R$ contains $U$. Let $\lambda>0$ be such that $W(x)=\mathbb E_x(e^{\lambda T_U})<+\infty$ for all $x$. Such a $\lambda$ exists since $\mu$ satisfies a Poincar\'e inequality. If $y \in \bar{H_R}$, $W(y) \leq K(R) <+\infty$ using the regularity of $W$. If $x \notin H_R$, we have $$\mathbb E_x\left(e^{\lambda T_U}\right)= \mathbb E_x\left(e^{ \lambda \, T_R} \, \mathbb E_{X_{T_R}}\left(e^{\lambda \, T_U}\right)\right) \leq K(R) \, \mathbb E_x\left(e^{ \lambda \, T_R}\right) \, .$$ But now for $\theta R> q \, \lambda $, 
\begin{eqnarray*}
W^q(x) &\leq& K^q(R)  \,\mathbb E^q_x\left(e^{ \lambda \, T_R}\right) \leq K^q(R) \, \mathbb E_x\left(e^{ q \, \lambda \, T_R}\right)\\ &\leq&  K^q(R) \, \mathbb E_x\left(e^{\int_0^{T_R} \, \theta \, h(X_s) \, ds}\right) \, = K^q(R) \, W_{\theta,U,h}(x) \, .
\end{eqnarray*}
The first part of the Theorem follows from Theorem \ref{thmgoodinteg}. The second part is immediate.
\end{proof}

\begin{remark}
If $\frac{d\nu}{d\mu} \in \mathbb L^1(\mu)$, La Vall\'ee-Poussin theorem implies that there is some Young function $\phi(u)=u\psi(u)$ with $\psi$ growing to infinity, such that $\frac{d\nu}{d\mu}$ belongs to the Orlicz space $\mathbb L_\phi$. Hence the second part of the Theorem will follow from the first one and the H\"{o}lder-Orlicz inequality, once $W \in \mathbb L_{\phi^*}(\mu)$. But for the previous proof to work we need something like $\phi^*(e^{\lambda u}) \leq C \, e^{\phi**(\lambda) \, u}$ which is not true for $\phi^*$ growing faster than a power function. That is why the result is only stated for $\nu$ with a density belonging to some $\mathbb L^p$ space. 
\end{remark}
\medskip

This result is nor new nor surprising. For instance, when a logarithmic Sobolev inequality holds true, Theorem 2.13 in \cite{CatGui3} shows that there is exponential convergence to the equilibrium in total variation distance as soon as $\nu$ belongs to the space $\mathbb L \, \ln \mathbb L$. It is then a simple consequence of Pinsker's inequality and the entropic convergence to equimibrium. More generally, the courageous reader will find in the jungle of section 3 of \cite{CatGui3} similar results for general $F$-Sobolev inequalities. It should be interesting to recover these results by using Lyapunov functions. 
\medskip

Actually we should describe the problem as follows: we know that reinforcing functional inequalities from Poincar\'e to $F$-Sobolev, reinforces the Lyapunov condition and conversely in the symmetric case. Does a reinforced functional inequality reinforce the integrability of the Lyapunov function and conversely in the symmetric case ? In particular can we characterize (at least in the symmetric case) a functional inequality through integrability properties of (some) Lyapunov function ?  

\begin{example}\label{exunifconv}
Look at the simple symmetric case $L=\Delta - \nabla V.\nabla$ in the whole $\mathbb R^d$. Then it is easily seen that the following holds: there is an equivalence between
\begin{itemize}
\item $x.\nabla V(x) \geq \alpha \, |x|^2$ for large $|x|$,
\item $W_e(x)= e^{\alpha \, |x|^2/2}$ is a Lyapunov function, i.e. satisfies $(HP1)$,
\item $W_2(x)= |x|^2$ also satisfies $(HP1)$.
\end{itemize}
The first item is of course equivalent to the fact that $V$ is uniformly convex. We see that the behavior of various Lyapunov functions can be very different. They also imply various integrability properties and functional inequalities. Notice however that thanks to what we said previously we can directly make the following reasoning: if $V$ is uniformly convex, $W_2$ is a Lyapunov function and admits some exponential moment so that convergence to the equilibrium in total variation distance holds as soon as $\nu$ belongs to the space $\mathbb L \, \ln \mathbb L$ as expected. Indeed, we have by $(HP1)$ that for some $\delta$ and some $r$ (denoting $B_r$ the euclidean ball of radius $r$, and $x\in B_r^c$
$$\mathbb E_x(e^{\delta T_{B_r}})\le  W_2(x).$$
Note however that in fact $W_e$ satisfies a stronger Lyapunov condition, i.e.
$$LW_e(x)\le -\lambda |x|^2\,W_e(x)+b1_{B_r}$$ so that by the previous result on logarithmic , we have a stroonger integrability property linked to hitting times:
$$\mathbb E_x\left(e^{\delta \int_0^{T_{B_r}}X_s^2}\right)\le  W_e(x).$$
However, it does not seem possible to pass from this last control to the previous one.\\ Hence, finally, these results do not give any precise idea of the dependence in $x$ of $\mathbb E_x(e^{\lambda T_U})$ for a given bounded $U$. Actually this is a very difficult problem for which results are only known in the gaussian case (i.e. for the Ornstein-Uhlenbeck process). \hfill $\diamondsuit$
\end{example}

In the next section we shall look more into details at the case where one can find a bounded Lyapunov function.
\bigskip

\section{Coming down from infinity, uniform geometric ergodicity and Lyapunov functions.}\label{secomdown}

If $d\nu/d\mu$ only belongs to $\mathbb L^1$, it is interesting to look at a bounded Lyapunov function. What precedes allows us to state
\begin{proposition}\label{propdown}
The following statements are equivalent:
\begin{itemize}
\item  there exists $\lambda >0$ such that $$\sup_{x \in \bar D} \, \mathbb E_x\left(e^{\lambda T_U}\right) \, < \, +\infty$$ for one (or all) non empty, open and bounded set $U$,
\item there exists a bounded Lyapunov function satisfying $(HP1)$,
\item the process is uniformly geometrically ergodic, i.e. there exist $\beta >0$ and $C>0$ such that $$\sup_{x \in \bar D} \, \Vert P_t(x,.) - \mu \Vert_{TV} \leq \, C \, e^{- \, \beta \, t} \, .$$
\end{itemize}
In this case of course, for any initial probability measure $\nu$, $$\Vert P_t(\nu,.) - \mu \Vert_{TV} \leq \, C \, e^{- \, \beta \, t} \, .$$
\end{proposition}
There exists a stronger form of uniform exponential integrability, the notion of ``coming down from infinity'' which is used by people who are studying quasi-stationary distributions or more precisely Yaglom limits (see e.g. the recent book \cite{CMSM}). We shall use the following definition
\begin{definition}\label{defcoming}
We say that the process comes down from infinity if for all $a>0$ there exists some open, bounded subset $U_a$ such that $\sup_x \, \mathbb E_x \left(e^{a \, T_{U_a}}\right) \, < \, + \infty$.
\end{definition}
In one dimension, this property was related to the uniqueness of quasi-stationary distributions (QSD) and to the fact that $\infty$ is an entrance boundary, in \cite{CCLMMSM}. Uniqueness of a (QSD) also follows from the ultraboundedness property of the semi-group, even in higher dimension (see e.g. \cite{CM}). In \cite{CGZ} Proposition 5.3, we claimed that ultraboundedness is actually equivalent to coming down from infinity for one dimensional diffusion processes with generator $\Delta -\nabla V.\nabla$ satisfying some extra condition. D. Loukianova pointed out to us that the proof of this proposition in fact needs slightly more stringent assumptions (the function $z \mapsto F(z)/z$ therein is not necessarily non-increasing) for this equivalence to hold true. 
\medskip

Nevertheless, part of this result is true and we shall give a direct and simple proof. 
\medskip

Indeed, assume that $W$ satisfies $(HP1)$. We have seen that we can always assume that $W \in \mathbb L^2(\mu)$ and then $W \in \mathcal D(L))$. It follows that 
\begin{equation}\label{eqpt}
L (P_t W) = P_t(L W) \leq - \lambda \, P_t W +  P_t(\mathbf 1_{\bar U}) \, .
\end{equation}
Now assume that $R \leq d(x,\bar U) \leq 2R$. We have
\begin{eqnarray*}
P_t(\mathbf 1_{\bar U})(x) &=& \mathbb E_x \left(\mathbf 1_{X_t \in \bar U}\right)\\ &\leq& \mathbb E_x \left(\mathbf 1_{T_{\bar U}<t}\right) \leq \sup_{d(y,\bar U) \geq R} \, \mathbb Q_y(T_{\bar U}<t)
\end{eqnarray*}
where $Q_y$ denotes the law of the process $Y_.$ with the same generator $L$ but reflected on $d(z,\bar U)=2R$ (which can be assumed to be smooth) and starting from $y$. Indeed if the process $X_.$ hits $\bar U$ before to leave $\{d(z,\bar U)\leq 2R\}$, it coincides with $Y_.$ (with the same starting point) up to $T_{\bar U}$. If not, $X_.$ leaves $\{d(z,\bar U)\leq 2R\}$ before $T_{\bar U}$, but in order to hit $\bar U$ it has to come back to $\{d(z,\bar U)\leq 2R\}$ first, so that using the Markov property we may apply the same argument as before this time starting from $X_{T_R}$ where $T_R$ is the hitting time of  $\{d(z,\bar U)\leq 2R\}$. That is why the final upper bound contains the supremum over $y$.\\ Now, since all coefficients are smooth, they are bounded with bounded derivatives of any order in $\{d(z,\bar U)\leq 2R\}$ which is compact. It is then well known that $$\sup_{d(y,\bar U)\geq R}\mathbb Q_y(T_{\bar U}<t) \leq C \, e^{-cR/t}$$ for some constants $C$ and $c$ only depending on these bounds. Hence $$P_t(\mathbf 1_{\bar U})(x) \leq C \, e^{-cR/t} \, ,$$ as soon as $R\leq d(x,\bar U) \leq 2R$. If $d(x,\bar U) > 2R$, $T_U>T_R$, and we may apply again the Markov property to get the same upper bound.\\ Pick some $R>0$ once for all and choose $t>0$ in such a way that $$C e^{-cR/t} < \frac 12 \, \lambda \, w \, .$$ We thus have
\begin{eqnarray*}
L (P_t W)(x) &\leq& - \lambda \, P_t W(x) + \mathbf 1_{d(x,\bar U)\leq R} + C e^{-cR/t} \, \mathbf 1_{d(x,\bar U)\geq R}\\ &\leq& - \frac{\lambda}{2} \, P_t W(x) + \mathbf 1_{d(x,\bar U)\leq R} \, ,
\end{eqnarray*}
so that $P_t W$ is a new Lyapunov function with $\lambda/2$ and $\{d(x,\bar U)\leq R\}$ in place of $\lambda$ and $\bar U$ (of course $P_t W$ belongs to $\mathcal D(L)$ and satisfies $P_t W \geq w >0$).\\ We deduce immediately
\begin{theorem}\label{thmultra}
Assume that $(HP1)$ is satisfied (for example $\mu$ satisfies a Poincar\'e inequality) and that the semi-group $P_.$ is ultra-bounded, i.e. that $P_t$ maps continuously $\mathbb L^1$ into $\mathbb L^\infty$ for any $t>0$. Then the process comes down from infinity. 
\end{theorem}
\begin{proof}
$(HP1)$ together with ultra-boundedness imply that the semi-group is hyper-contractive so that $\mu$ satisfies a logarithmic Sobolev inequality according to Gross theorem. Hence, we may apply Corollary \ref{corlslyap} and find a Lyapunov function $W$ such that $$LW \leq - \lambda \, d^2(.,x_0) \, W + b \, \mathbf 1_{\bar U}$$ for some bounded open subset $U$. Hence for all large enough $a>0$, $$LW \leq - 2a \, W \, + b_a \, \mathbf 1_{U_a}$$ with $U_a = \{x \, ; \, \lambda \, d^2(x,x_0) \, \leq 2a \}$ and $b_a= \sup_{x \in U_a} \, LW$, $a$ being large enough for $\bar U \subset U_a$. According to the previous discussion there exists a new Lyapunov function $W_a=P_tW$ for some adequat $t$ which is bounded and satisfies $LW_a(x) \leq - aW_a(x)$ for $x \in V_a=\{d(x,U_a)>R\}$, so that $$\sup_x \, \mathbb E_x \left(e^{aT_a}\right) < +\infty \quad \textrm{ for $T_a$ the hitting time of $V_a^c$.}$$
\end{proof}
\begin{remark}
The fact that $P_t W$ is still a Lyapunov function is interesting by itself, but except in the ultra-bounded situation, it does not furnish new results. For instance if we want to get a Lyapunov function that belongs to all the $\mathbb L^p$ spaces, we have to assume that $P_.$ is immediately hyper bounded, which is stronger than log-Sobolev, while Theorem \ref{thmconvvartot} gives the result under a simple $F$-Sobolev condition. Working a little bit more, one can extend the previous result to discrete valued Markov process, which are ultracontractive (birth-death processes,...).
\end{remark}

The previous proof indicates a way to prove ``coming down from infinity'' using Lyapunov functions, more precisely nested Lyapunov conditions
\begin{definition}
\item[$ (SLC)$] We shall say that a Super-Lyapunov condition is satisfied if  there exist a sequence $W_k\ge 1$, a sequence of increasing bounded sets $B_k$ growing to $\R^n$, an increasing sequence $\lambda_k>0$ growing to infinity and a sequence $b_k>0$ such that
$$LW_k\le -\lambda_k \, W_k +b_k \mathbf 1_{B_k}.$$
\end{definition}
For instance in the situation of Theorem \ref{thmgoodinteg}, $(SLC)$ is satisfied with the same Lyapunov function $W$, i.e. $W=W_k$, similarly as what we have done (in the case of log-Sobolev) in the previous proof. Of course if all the $W_k$ are bounded, $(SLC)$ is equivalent to ``coming down from infinity''.
\medskip

Assume that $(SLC)$ is satisfied. Denote by $T_k$ the hitting time of $B_k$ and choose some $\delta>0$. As we did in the proof of Theorem \ref{thmconvvartot}, for $k\geq k_0$ and $\lambda_{k_0}>\delta$ we have for $x \notin B_{k}$,
\begin{eqnarray*}
\mathbb E_x\left(e^{\delta T_{k_0}}\right)&=& \mathbb E_x\left(e^{ \delta \, T_{k}} \, \mathbb E_{X_{T_k}}\left(e^{\delta \, T_{k_0}}\right)\right) \leq \, \mathbb E_x\left(e^{ \delta \, T_k}\right) \sup_{y \in \partial B_k} \, \mathbb E_y\left(e^{\delta T_{k_0}}\right)\\ &\leq& 
\left(\mathbb E_x\left(e^{ \lambda_k \, T_k}\right)\right)^{\delta/\lambda_k} \,  \sup_{y \in \partial B_k} \, \mathbb E_y\left(e^{\delta T_{k_0}}\right) \\ &\leq& (W_k(x))^{\delta/\lambda_k} \, \sup_{y \in \partial B_k} \, \mathbb E_y\left(e^{\delta T_{k_0}}\right) \, .
\end{eqnarray*}
Define $$w_{k}= \sup_{y \in B_{k+1}-B_k} \, W_k(y) \, .$$
Proceeding by induction we have for all $x\notin B_{k_0+1}$
\begin{equation}\label{equnifgeom}
\mathbb E_x\left(e^{\lambda T_{k_0}}\right) \leq  C_{k_0} \prod_{k=k_0}^{+\infty} (w_k)^{\frac\delta{\lambda_k}}.
\end{equation}
As we see, if the sequence $\lambda_k$ goes to infinity, $\delta$ has no role in the convergence of the previous infinite product, as well as the value of $k_0$. We thus have
\begin{theorem}\label{thmslc}
If $(SLC)$ is satisfied, and for some $k_0$, $$\sum_{k=k_0}^{+\infty} \, \frac{\ln w_k}{\lambda_k}  < +\infty \, ,$$ the process comes down from infinity. Hence $(SLC)$ is satisfied with another sequence $(\lambda_k,W_k,B_k)$ where all the $W_k$ are bounded.
\end{theorem}
\begin{example}
In one dimension consider (for $\beta>0$) the potential $$V_\beta(x)=(1+x^2) \, \ln^\beta(1+x^2) \, .$$ For $W(x) = e^{x^2/2}$ it holds for $|x|$ large enough, $$L_\beta W(x) \leq - \, (1-\varepsilon) \, x^2 \, \ln^\beta(1+x^2) \, W(x)\, .$$ Hence if $B_k=B(0,R_k)$ we have $\ln (w_k) \leq \frac 12 \, R_k^2$, while for $|x|>R_k$, we may choose $\lambda_k = c \, R_k^2 \, \ln^\beta(R_k)$. It follows that the process comes down from infinity for any $\beta>0$ by choosing for instance $R_k=\exp k^{2/\beta}$  in Theorem \ref{thmslc}. \\ It is known however that the semi-group is ultra-bounded if and only if $\beta>1$ (see e.g \cite{KKR}). Hence we have some examples of processes coming down from infinity for which the semi-group is not ultra-bounded (but is immediately hyper-contractive as shown in \cite{KKR}). 
\end{example}
\bigskip

\section{Integrability and Lyapunov conditions.}\label{secintegral}

In this section we will study the interplay between Lyapunov conditions and integrability conditions. So it is not a restriction to look at $L=\Delta - \nabla V.\nabla$ such that $\mu$ is symmetric (or reversible).

Let us recall that integrability conditions are related to transportation inequalities. Indeed, if $\mu$ satisfies a Gaussian integrability condition, i.e.
$$\exists \delta>0, x_0\qquad \int e^{\delta d(x,x_0)^2}d\mu<\infty$$
then 
$$(T1) \quad \forall \nu\in {\mathcal P}_1,\qquad W_1(\nu,\mu)\le \sqrt{2C H(\nu|\mu)}$$
for some explicit $C$ (see \cite{DGW}), where $W_1$ and $H$ denote respectively the $1$-Wasserstein distance and  the relative entropy (or Kullback-Leibler information). The converse, $(T1)$ implies gaussian integrability is also true and due to K. Marton \cite{Mar96}. Gozlan \cite{Goz06} has generalized the approach to various type of $(T1)$ type transportation inequalities, getting that integrability property 
$$\exists \delta>0, x_0\qquad \int e^{ \alpha(\delta d(x,x_0))}d\mu<\infty$$
for $\alpha$ convex (quadratic near 0) is equivalent to $\alpha (W_1(\nu,\mu)\le C\,H(\nu|\mu)$ for very $\nu$. The links between functional inequalities and integrability properties are usually consequences of concentration properties of lipschitzian function (as Herbst argument for logarithmic Sobolev inequality). A more direct approach is taken in \cite{CatGui1}, where in particular an equivalence between a Poincar\'e inequality and a weak form of a $(T2)$ inequality (involving $W_2$) is shown. Note however that we won't require here some local inequality, so that we cannot hope to get better functional inequalities than $(T1)$ like.
\medskip

We shall now investigate the relationship between the existence of Lyapunov functions and integrability properties. In the sequel we will thus assume the following 
\begin{definition}\label{defphilyap}
Let $\phi$ be a (strictly) positive function. We say that $(\phi-Lyap)$ is satisfied if there exist $W > 0$, $b\geq 0$ and a bounded open subset $C$ such that  $$LW \leq - \, \phi^2 \, W + b \, \mathbf 1_C \, .$$
\end{definition} 
If we define $U = \ln(W)$ it is immediately seen that 
\begin{equation}\label{eqlog}
LU \, + \, |\nabla U|^2 \, \leq \, - \phi^2 + \, \frac{b}{\min_{\bar C}W} \, \mathbf 1_C \, .
\end{equation}
Using this inequality it is shown in \cite{CGW} lemma 3.4 that for any smooth $h$,
\begin{equation}\label{cruc}
\int h^2 \, \phi^2 d\mu \, \leq \, \mathcal E(h) \, + \, \frac{b}{\min_{\bar C}W} \, \int_C \, h^2 \, d\mu \, .
\end{equation}
Of course, if $\int_C \, h \, d\mu=0$, and provided $C$ is smooth enough (which is not a restriction), we may apply the Holley-Stroock perturbation argument and get $$\int h^2 \, \phi^2 d\mu \, \leq \, \left(1+\frac{b}{\min_{\bar C}W} \, e^{Osc_{\bar C}V} \, C_P(C,dx)\right) \mathcal E(h)$$ where $Osc_U(V)$ denotes the oscillation of $V$ on the subset $U$ and $C_P(U,dx)$ the Poincar\'e constant of the uniform measure on $U$. If $\phi$ goes to infinity at infinity, the previous inequality is thus stronger than the Poincar\'e inequality (as expected). It is well known that the Poincar\'e inequality implies the exponential integrability of $c \, d(.,x_0)$ for some small enough positive $c$, as the logarithmic Sobolev inequality implies the exponential integrability of $c \, d^2(.,x_0)$ for some small enough positive $c$ (we previously recalled this result). 
\medskip

In \cite{liu1}, Yuan Liu proved that $(\phi-Lyap)$ with $\phi(x)=a \, d(x,x_0)$ implies the exponential integrability of $c \, d^2(.,x_0)$ for $c < a$. This result is not surprising since in this case $(\phi-Lyap)$ is exactly $(HLS1)$ which is equivalent to the logarithmic-Sobolev inequality, at least if the curvature is bounded from below according to Theorem \ref{thmls}, hence implies as a consequence the quoted exponential integrability. 

\noindent We shall here follow and generalize Liu's argument in order to answer the following question: what are sufficient conditions on $\psi$ for $c \, \psi^2$ to be exponentially integrable when $(\phi-Lyap)$ is satisfied ?

\noindent A particularly interesting example is the case when $\phi(x)=a \, d^p(x,x_0)$. Indeed, for $p=0$ we know that a Poincar\'e inequality is satisfied and for $p=1$ a log-Sobolev inequality is satisfied under a curvature assumption. The use of curvature is very specific and strongly connected to $p=1$. Hence for $0<p<1$ we do not know whether $(\phi-Lyap)$ implies some natural functional inequality or not, even for bounded from below curvature. We shall see however that $c \, d^{p+1}(.,x_0)$ is exponentially integrable for some small enough positive $c$ and all $0 \leq p$, so that a generalized transportation inequality holds. 
\bigskip

For $\psi\ge 0$, introduce
$$\beta_n:=\int \psi^{2n}d\mu \, .$$
We will use \eqref{cruc} to initiate a recurrence on $\beta_n$ using the notation $\bar b= \frac{b}{\min_{\bar C}W}$,
\begin{eqnarray*}
\beta_n&=&\int \frac{\psi^{2n}}{\phi^2}\phi^2d\mu\\
&\le& \int\left|\nabla\left(\frac{\psi^n}\phi\right)\right|^2d\mu+\bar b\int_C\frac{\psi^{2n}}{\phi^2}d\mu.
\end{eqnarray*} 
Let us focus on the first term
\begin{eqnarray*}
 \int\left|\nabla\left(\frac{\psi^n}\phi\right)\right|^2d\mu&=&\int\left|n\psi^{n-1}\frac{\nabla \psi}{\phi}-\psi^n\frac{\nabla \phi}{\phi^2}\right|^2 \, d\mu\\
 &=&n^2\int\frac{\psi^2|\nabla \psi|^2}{\phi^2}\,\psi^{2(n-2)}+2n\int \frac{\psi\nabla\phi.\nabla\psi}{\phi^3}\psi^{2(n-1)}d\mu
 +\int \frac{|\nabla \phi|^2}{\phi^4}\psi^{2n}d\mu.
\end{eqnarray*}

Let us assume then that there exists $\alpha,\beta,\gamma>0$ and $0<\delta<1$, (at least outside a compact $K$, and if so choose $\psi$ to be 0 on $K$) such that
\begin{equation}\label{eqcondition}
\frac{\psi^2|\nabla \psi|^2}{\phi^2}\le \alpha,\quad\left|\frac{\psi\nabla\phi.\nabla\psi}{\phi^3}\right|\le \beta,\quad \frac{|\nabla \phi|^2}{\phi^4}\le\delta,\quad \sup_C\frac{\psi^2}{\phi^2}\le \gamma \, .
\end{equation}
Under these assumptions, we get that
$$\beta_n\le \frac{\alpha}{1-\delta} n^2\beta_{n-2}+\frac{2n\beta+\gamma \bar b}{1-\delta}\beta_{n-1} \, .$$
Combined with a direct consequence of Cauchy-Schwarz inequality, we obtain
$$\beta_n\le \sqrt{\beta_{n+1}\beta_{n-1}}\le\left(\frac{\alpha}{1-\delta} (n+1)^2\beta_{n-1}+\frac{2(n+1)\beta+\gamma \bar b}{1-\delta}\beta_n\right)^{\frac12}\beta_{n-1}^{\frac12}.$$
We then easily deduce that$$
\beta_n\le\frac12\left(\frac{2(n+1)\beta+\gamma \bar b}{1-\delta}+\sqrt{\left(\frac{2(n+1)\beta+\gamma \bar b}{1-\delta}\right)^2+4\frac{\alpha}{1-\delta} (n+1)^2}\right)\beta_{n-1}\le a n \, \beta_{n-1}$$
for
$$a>\frac12\left(\frac{2\beta}{1-\delta}+\sqrt{\frac{4\beta^2}{(1-\delta)^2}+4\frac{\alpha}{1-\delta}}\right)$$ and $n$ large enough. It then follows that for some $c$
$$\beta_n\le ca^n n!$$
and thus we have that for $a'<a^{-1}$
$$\int e^{a' \psi^2}d\mu \le \frac {c}{1-a'a}.$$

Now come back to \eqref{eqcondition}. I f we assume that $\phi$ is bounded from below by some positive constant in $C$, the condition on $\gamma$ is satisfied as soon as $\phi$ and $\psi$ are, say, continuous. The condition on $\delta$ says that $1/\phi$ is a contraction outside some compact set. Assuming this condition we see that both conditions on $\alpha$ and $\beta$ are the same, i.e. $|\nabla (\psi^2)|/\phi$ is bounded. Thus
 
\begin{theorem}\label{thminteg}
Assume that $(\phi-Lyap)$ is satisfied for some function $\phi$ such that $\phi$ is bounded from below by a positive constant on $C$ and $1/\phi$ is $\eta$-Lipschitz for some $\eta<1$. Then for all function $\psi^2$ such that $|\nabla (\psi^2)|/\phi$ is bounded, there exists $a'>0$ such that $\int e^{a' \psi^2}d\mu < +\infty$.
\end{theorem}

\begin{example}
\begin{enumerate}
\item For a constant $\phi$ we recover the fact that any Lipschitz function has an exponential moment, hence the usual concentration result once a Poincar\'e inequality is satisfied. For $\phi(x)= a d^2(x,x_0)$, we recover the gaussian nature of the tails once a log-Sobolev inequality is satisfied.
\item If we take $\phi(x)= a d^p(x,x_0)$ for $p\geq 0$ we obtain exponential integrability of functions $g$ such that $|\nabla g| \leq C \, |x|^{p}$, hence for instance for $g(x)=d^{p+1}(x,x_0)$. As for the Poincar\' e or the log-Sobolev case, this exponential integrability is sharp, since for $\mu(dx)=Z_p \, e^{-|x|^{p+1}}$, $(\phi-Lyap)$ is satisfied with $\phi(x)=c|x|^p$ and $W(x)= e^{a |x|^p}$ for a small enough $a$.
\item One may of course also consider $\phi^2(x)=ad^{-1}(x,x_0)$, if $a>1$, which can be obtained for Cauchy type measures (see \cite{CGGR} for details). In this case we may take $\psi^2$ behaving as $\log(|x|)$ at infinity recovering polynomial integrability (also quite sharply). Note also that in \cite[Th. 5.1]{CGGR}, it is shown how a converse Poincar\'e inequality (obtained by a $(\phi-Lyap)$ condition and a local Poincar\'e inequality) plus an involved integrability condition implies a weak Poincar\'e inequality. This integrability condition can be checked using Th. \ref{thminteg} for example for Cauchy type measure (via tedious computations), so that only a Lyapunov condition and local inequality are also sufficient for weak Poincar\'e inequality.
\end{enumerate}
\end{example}

Of course there is no converse statement for Theorem \ref{thminteg}, since for instance exponential integrability for the distance cannot imply a Poincar\'e inequality (disconnected domains for example).
\medskip

Finally recall that that $(\phi-Lyap)$ is equivalent to the following $$\mathbb E_x \left(e^{\int_0^{T_C} \, \phi^2(X_s) \, ds}\right) < +\infty \, ,$$ where in our situation $$X_t = X_0 + \sqrt 2 \, W_t - \int_0^t \, \nabla V(X_s) \, ds \, .$$
It would be particularly interesting to show that weak Poincar\'e inequality implies back Lyapunov condition as done in section 2 for Poincar\'e inequality.

\bibliographystyle{plain}
\bibliography{cg--}

\end{document}